\documentclass[12pt]{article}
\usepackage[centertags]{amsmath}
\usepackage{amsfonts}
\usepackage{amssymb}
\usepackage{amsthm}
\usepackage{amsmath,mathrsfs}
\usepackage{amsmath,amscd}
\usepackage{mathtools}
\usepackage[matrix,arrow,curve,cmtip]{xy}
\title{\textbf{A Dual Representation in Spectral Algebraic Geometry}}
\author{Renaud Gauthier \footnote{rg.mathematics@gmail.com} \\ \\}
\theoremstyle{definition}
\newtheorem{AffPerc}{Theorem}[section]
\newtheorem{Perc}[AffPerc]{Theorem}
\newtheorem{DualitySpDM}{Theorem}[section]

\newtheorem{DualityFctr}[DualitySpDM]{Theorem}

\newtheorem{Var1}{Theorem}[subsection]
\newtheorem{Var2}[Var1]{Theorem}

\DeclareMathOperator*{\colim}{\text{colim}}

\newcommand{\beq}{\begin{equation}}
\newcommand{\eeq}{\end{equation}}
\newcommand{\ind}{\rightsquigarrow}
\newcommand{\hrarr}{\hookrightarrow}
\newcommand{\rarr}{\rightarrow}
\newcommand{\Lrarr}{\longrightarrow}

\newcommand{\Ob}{\text{Ob}}
\newcommand{\surj}{\twoheadrightarrow}
\newcommand{\xrarr}{\xrightarrow}

\newcommand{\cA}{\mathcal{A}}

\newcommand{\cC}{\mathcal{C}}

\newcommand{\cD}{\mathcal{D}}
\newcommand{\cE}{\mathcal{E}}
\newcommand{\cF}{\mathcal{F}}
\newcommand{\cG}{\mathcal{G}}
\newcommand{\cH}{\mathcal{H}}

\newcommand{\cO}{\mathcal{O}}

\newcommand{\cS}{\mathcal{S}}
\newcommand{\cT}{\mathcal{T}}

\newcommand{\cX}{\mathcal{X}}
\newcommand{\cY}{\mathcal{Y}}
\newcommand{\cZ}{\mathcal{Z}}

\newcommand{\gm}{\mathfrak{m}}

\newcommand{\bE}{\mathbb{E}}

\newcommand{\bZ}{\mathbb{Z}}
\newcommand{\Cat}{\text{Cat}}
\newcommand{\Catinf}{\Cat_{\infty}}
\newcommand{\CatDinf}{\Cat^{\Delta}_{\infty}}

\newcommand{\Fun}{\text{Fun}}
\newcommand{\hlim}{\text{hlim}}
\newcommand{\Hom}{\text{Hom}}

\newcommand{\Ind}{\text{Ind}}

\newcommand{\Map}{\text{Map}}

\newcommand{\Mod}{\text{Mod}}

\newcommand{\op}{\text{op}}

\newcommand{\Set}{\text{Set}}
\newcommand{\SetD}{\Set_{\Delta}}

\newcommand{\Spec}{\text{Spec}}

\newcommand{\QCoh}{\text{QCoh}}

\newcommand{\Sh}{\text{Sh}}

\newcommand{\Aff}{\text{Aff}}
\newcommand{\AffX}{\Aff_{X}}
\newcommand{\AffXop}{\AffX^{\op}}
\newcommand{\Alg}{\text{Alg}}
\newcommand{\Aa}{\text{A}_{\alpha}}

\newcommand{\bfX}{\mathbf{X}}
\newcommand{\bX}{\textbf{X}}

\newcommand{\bfY}{\mathbf{Y}}
\newcommand{\CRing}{\text{CRing}}
\newcommand{\Ctensor}{\cC^{\otimes}}

\newcommand{\CAlg}{\text{CAlg}}
\newcommand{\CAlgdiscr}{\CAlg^{\text{discr}}}
\newcommand{\CAlgcn}{\CAlg^{\text{cn}}}
\newcommand{\CAlgaugk}{\CAlg^{\text{aug}}_k}
\newcommand{\CAlgaugkop}{(\CAlgaugk)^{\op}}
\newcommand{\CAlgartk}{\CAlg^{\text{art}}_k}
\newcommand{\CAlgR}{\CAlg_R}
\newcommand{\CAlgetR}{\CAlgR^{\text{\'et}}}
\newcommand{\CAlgk}{\CAlg_k}

\newcommand{\cOX}{\cO_{\cX}}
\newcommand{\cOXUa}{\cOX|_{\Ua}}
\newcommand{\cOY}{\cO_{\cY}}
\newcommand{\cOXU}{\cOX|_U}
\newcommand{\cOSpfR}{\cO_{\Spf R}}
\newcommand{\cOSpetR}{\cO_{\Spet R}}

\newcommand{\cXUa}{\cX_{/\Ua}}
\newcommand{\cXU}{\cX_{/U}}
\newcommand{\Eop}{\cE^{\op}}
\newcommand{\Ek}{\mathbb{E}_k}
\newcommand{\Einf}{\mathbb{E}_{\infty}}

\newcommand{\FunEhatS}{\Fun(\cE, \hat{\cS})}
\newcommand{\FunCAlgcnhatS}{\Fun(\CAlgcn, \hat{\cS})}
\newcommand{\FunCAlgartkS}{\Fun(\CAlgartk, \cS)}
\newcommand{\FunCAlgcnS}{\Fun(\CAlgcn, \cS)}
\newcommand{\FunECD}{\Fun_{\cE}(\cC, \cD)}
\newcommand{\FunEopS}{\Fun(\cE^{\op}, \cS)}

\newcommand{\FuncartECD}{\Fun^{\text{cart}}_{\cE}(\cC, \cD)}

\newcommand{\FuncocartCAlgartk}{\Fun^{\text{cocart}}_{\CAlgartk}}
\newcommand{\FuncocartCAlgcn}{\Fun^{\text{cocart}}_{\CAlgcn}}
\newcommand{\Fins}{\cF\text{in}_*}
\newcommand{\ggs}{\mathfrak{g}_*}
\newcommand{\hCatinf}{\widehat{\Catinf}}
\newcommand{\is}{\iota^*}
\newcommand{\isX}{\is \bfX}
\newcommand{\infTopsHenMod}{\infty \cT \text{op}_{\Mod}^{\text{sHen}}}
\newcommand{\infTopsHenCAlg}{\infty \cT \text{op}_{\CAlg}^{\text{sHen}}}
\newcommand{\infTopMod}{\infty\cT\text{op}_{\Mod}}

\newcommand{\Kan}{\text{Kan}}
\newcommand{\LX}{\text{L}_{\bfX}}
\newcommand{\LXY}{\text{L}_{\bfX/\bfY}}
\newcommand{\Liek}{\text{Lie}_k}
\newcommand{\LMod}{\text{LMod}}
\newcommand{\LModR}{\LMod_R}
\newcommand{\ModX}{\Mod^{\bfX}}
\newcommand{\Modulik}{\text{Moduli}_k}
\newcommand{\Modk}{\Mod_k}
\newcommand{\ModA}{\Mod_A}
\newcommand{\ModXart}{\ModX_{\text{art}}}
\newcommand{\ModXcn}{\ModX_{\text{cn}}}
\newcommand{\ModisX}{\Mod^{\isX}}
\newcommand{\ModisXart}{\ModisX_{\text{art}}}

\newcommand{\oT}{\otimes}

\newcommand{\bQCoh}{\textbf{QCoh}}
\newcommand{\bQCohart}{\bQCoh_{\text{art}}}
\newcommand{\Rep}{\text{Rep}}
\newcommand{\Repgs}{\text{Rep}_{\ggs}}
\newcommand{\ShadR}{\Sh^{\text{ad}}_R}
\newcommand{\SetDRE}{(\SetD^{\text{Rfib}})_{/\cE}}
\newcommand{\ShetR}{\text{Sh}^{\text{\'et}}_R}
\newcommand{\Spet}{\text{Sp\'et}}
\newcommand{\SpDM}{\text{SpDM}}
\newcommand{\Sp}{\text{Sp}}
\newcommand{\Spf}{\text{Spf}}
\newcommand{\Sfins}{\cS^{fin}_*}
\newcommand{\tN}{\text{N}}
\newcommand{\tcn}{\text{cn}}
\newcommand{\tauleqn}{\tau_{\leq n}}

\newcommand{\Ua}{U_{\alpha}}

\begin{document}
\maketitle
\begin{abstract}
	Given a spectral Deligne-Mumford stack $X$, we define a perception of $X$ to be a collection of a certain class of morphisms $Y \rarr X$. For the class of affine morphisms in $\SpDM_{/X}$, we show that from $\QCoh(X)$ one can extract the affine perception $\AffX$ of $X$ on the one hand, and a subcategory of an $\infty$-category of representations $\Repgs$ of a dg Lie algebra $\ggs$ associated with $X$ on the other. For the class of local morphisms $\Spet R \rarr X$, the local perception of $X$ is given by the functor $\bfX = \Hom(\Spet(-), X)$ it represents. If $\bfX$ is a geometric stack, Tannaka duality allows us to recover $\bX$ from $\bQCoh(\bfX)$, from which we can also get, after base change, a subcategory of $\Repgs$. We generalize those results by considering functors $\bfX: \CAlgcn \rarr \cS$ that are representable in accordance with the spectral Artin representability theorem of Lurie. 
\end{abstract}

\newpage

\section{Introduction}
In this very short paper, we show that for a class of geometric objects $X$, we have a connection, in a sense to be precised below, between certain collections of maps into $X$, which provides what we call the \textbf{perception} of $X$, and the representation theory of those objects $X$. More generally, if $\bX$ satisfies the hypotheses of the spectral Artin representability theorem (Theorem 16.0.1 of \cite{SAG}), it is representable by a spectral Deligne-Mumford stack $X$: $\bX = \Hom(\Spet(-), X)=h_X$. Further by Proposition 6.2.4.1 of \cite{SAG}, we have an equivalence of $\infty$-categories $\QCoh(h_X) \simeq \QCoh(X)$. Finally by Proposition 2.5.1.2 of \cite{SAG}, for $\AffX$ the full subcategory of the slice category $\SpDM_{/X}$ of spectral Deligne Mumford stacks over $X$ spanned by affine maps, we have an equivalence of $\infty$-categories $\AffXop \simeq \CAlg(\QCoh(X)^{\tcn}) $. Collecting things together, we arrive at $\AffXop \simeq \CAlg(\bQCoh(\bfX)^{\tcn})$ (we use boldface notations for concepts pertaining to functors). We refer to $\AffX$ as the \textbf{affine perception} of $X$. 
On the other hand, Thm 13.4.0.1 of \cite{SAG} states that for a field $k$ of characteristic zero, $\bX: \CAlgartk \rarr \cS$ a formal moduli problem over $k$, $\Psi^{-1}(\bX) = \ggs$ its associated dg Lie algebra over $k$, we have a fully faithful monoidal embedding $\bQCohart(\bX) \hrarr \Repgs$. Here $\Psi: \Liek \rarr \Modulik$ is an equivalence of $\infty$-categories provided by Thm 13.0.0.2 of the same reference, and $\bQCohart$ is the quasi-coherent sheaf functor on formal moduli problems. In this context, we refer to the $\infty$-category of quasi-coherent sheaves on $\bX$ as the \textbf{manifestation} of $\bX$, not to be confused with the manifestation as introduced in \cite{SAG}, and $\Repgs$ as the \textbf{representation} mentioned above. The connection between perceptions and representations is provided by the obvious functor $\iota: \CAlgartk \rarr \CAlgcn$. For the pullback of the artinian object $\bX: \CAlgcn \rarr \cS$ along $\iota$, we have a fully faithful embedding as pointed out above:
\beq
\bQCohart(\isX) \hrarr \Rep_{\Psi^{-1} (\isX)} \nonumber
\eeq
where functors $\bfX: \CAlgcn \rarr \cS$ such that $\iota^* \bfX \in \Modulik$ are said to be artinian. The connection between quasi-coherent sheaves on a functor $\bX$ and $\isX$ is simply given by base change:
\beq
\bQCohart(\isX) = \CAlgartk \times_{\CAlgcn} \bQCoh( \bfX) \nonumber
\eeq

Collecting things together, we have our first result:

\newpage

\begin{AffPerc}
	If a functor $\bX: \CAlgcn \rarr \cS$ satisfies the hypotheses of the spectral Artin representability theorem, and $\bfX = \Hom(\Spet(-), X)$ for some spectral Deligne-Mumford stack $X$, if further $\bfX$ is artinian, then from the manifestation $\bQCoh(\bfX)$ of $\bfX$ one can extract a perception as $\CAlg(\bQCoh(\bfX)^{\tcn}) \simeq \AffXop$, and one also has a representation theoretic presentation of $\bQCoh(\bfX)$ as a full subcategory of $\Rep_{\Psi^{-1}(\isX)}$ after base change. 
\end{AffPerc}
For a stronger result, we use Theorem 9.3.0.3 of \cite{SAG}, Tannaka duality for geometric stacks: if $\bX$ is a geometric stack, then $\bX$ can be functorially recovered from $\bQCoh(\bX)$. In this situation, if $\bfX$ is representable, $\bfX = \Hom(\Spet(-), X)$, $\bfX$ would correspond to a functorial local perception of $X$. It follows that we have the following result:
\begin{Perc}
	If $\bX$ is an artinian geometric stack, and satisfies the hypotheses of the spectral Artin representability theorem, then from $\bQCoh(\bfX)$ one can recover the local perception $\bfX$ of the spectral Deligne-Mumford stack it represents, and one can obtain a full subcategory of $\Rep_{\Psi^{-1}(\isX)}$ after base change.
\end{Perc}

In both instances, the manifestation $\bQCoh(\bfX)$ of $\bfX$ is pivotal in obtaining the perception of $X$ and a subcategory $\bQCohart(\isX)$ with a representation theoretic flavor, a dual representation of sort. This provides a bridge between how a geometric object $X$ is perceived, and its actions through its representation theory.\\

Most of this work is essentially based on three references; Higher Topos Theory, Higher Algebra, and Spectral Algebraic Geometry, all by the same author, J. Lurie. Rather than referring the reader to those voluminous references for various results, we thought it necessary to make the present work reasonably self-contained. By this we mean that we will re-introduce only that material which is immediately connected to what we are discussing. Background topics such as $\infty$-categories, $\infty$-topos, $\infty$-operads, etc... will not be covered, and will be assumed to be well understood. We have chosen to follow a strictly utilitarian, though unorthodox, way of presenting the requisite material, that is in a top down manner, by presenting the most advanced concepts first, and those they depend upon afterwards, rather than following the conventional bottom up approach, which would leave most readers wondering why are certain concepts introduced. This will also allow readers to pick what they need in the initial review. The reason for doing this also is that we claim no originality about the background material. Various pointed references are given as footnotes; putting references in the body of the text would make it unreadable.\\

Special notations: $\ind$ will mean induces, $\surj$ is used for surjective maps, CD stands for commutative diagram, RLP means right lifting property, L is used for left, R for right.\\

References are standard. For background material on Algebraic Geometry one can use \cite{AG}. For Topos Theory we use \cite{SHT} and \cite{HTT}. For $\infty$-categories, \cite{SHT} for foundations, and \cite{HTT}. For $\infty$-operads, symmetric monoidal $\infty$-categories and $\Einf$-ring spectra, \cite{HA}. Spectral Algebraic Geometry is developed in \cite{SAG}.\\

The first six sections are review material; we very briefly cover $\Einf$-ring spectra, spectral Deligne-Mumford stacks, the spectral Artin representability theorem and the representation theory of formal moduli problems. The reader who feels comfortable with these topics can safely skip those sections and jump to section $7$ in which we start new material. Note that whenever we discuss formal moduli problems, a field $k$ of characteristic zero is assumed to having been fixed.\\

\section{Spectra}
Central to Spectral Algebraic Geometry are $\Einf$-rings. Those are objects of the category $\Sp$ of spectra. We first define spectra, and then all accompanying notions.\\

A \textbf{spectrum} \footnote{Def. 1.4.2.8 \cite{SAG}} is a \textbf{reduced} \footnote{Def.1.4.2.1 \cite{SAG}} (maps pushouts to pullbacks) and \textbf{excisive} \footnote{Def. 1.4.2.1 \cite{SAG}} (preserves final objects) functor $X: 
\Sfins \rarr \cS$. We denote by $\Sp$ the $\infty$-category of spectra.\\

In this definition, $\Sfins$ is the smallest full subcategory of $\cS_*$, the $\infty$-category of pointed objects of $\cS$, which is stable under finite colimits. Here $\cS$ refers to the $\infty$-category of spaces, which is defined as the simplicial nerve of the full subcategory $\Kan$ of $\SetD$ spanned by Kan complexes: $\cS = \text{N}(\Kan)$ \footnote{Def. 1.2.16.1 \cite{HTT}}.\\

Another (rather indirect way) of defining spectra goes as follows, as covered in \cite{SAG}. If $\cH$ denotes the category of pointed CW complexes with homotopy classes of pointed maps between them as morphisms, we have a suspension functor $\Sigma: \cH \rarr \cH$. Observing that $\cH = h \Sfins$, the suspension functor lifts in $\Sfins$ to a functor that we denote by the same letter, giving rise to a sequence $\cdots \rarr \Sfins \xrarr{\Sigma} \Sfins \cdots$, whose limit is denoted $\Sp^{fin}$, and the $\infty$-category of spectra is defined as $\Sp = \Ind(\Sp^{fin})$, the $\infty$-category of Ind objects of $\Sp^{fin}$ obtained by formally adding filtered colimits. \footnote{ Constr. 0.2.3.10 \cite{SAG}}

\section{$\Einf$-ring spectra}
$\Einf$-rings are objects of the $\infty$-category $\CAlg = \CAlg(\Sp)$ of commutative algebra objects of $\Sp$, where $\Sp$ is regarded as a symmetric monoidal $\infty$-category for the smash product monoidal structure.\\

For $k \geq 0$, we define a \textbf{$\Ek$-ring} \footnote{Def.7.1.0.1. \cite{HA}}, to be an object of the $\infty$-category $\Alg_{\Ek}(\Sp)$. Independently, objects of the $\infty$-category $\Alg_{\tN(\Fins)}(\Sp) = \CAlg(\Sp) = \CAlg$ are referred to as commutative algebra objects of $\Sp$, and this $\infty$-category $\CAlg$ can be identified with the homotopy limit $\hlim \Alg_{\Ek}(\Sp)$. Objects thereof are called \textbf{$\Einf$-rings}. For $k \geq 0$, let $R$ be a $\bE_{k+1}$-ring. Let $\Alg^{(k)}_R = \Alg_{\Ek}(\LModR)$ be the \textbf{$\infty$-category of $\Ek$-algebras over $R$} \footnote{Def. 7.1.3.5 \cite{HA}}, where $\LModR = \LMod_R(\Sp)$ is regarded as a $\bE_{k+1}$-monoidal $\infty$-category. We also denote by $\CAlgR = \CAlg(\LModR) = \Alg_{\Einf}(\LModR)$ the \textbf{$\infty$-category of $\Einf$-algebras over $R$} \footnote{Variant 7.1.3.8 \cite{HA}}. \\

We now go over a few properties of $\Einf$-rings. For $n \in \bZ$ define the $n$-sphere to be $S^n = (S^0, n)$ \footnote{Def. 0.2.3.2. \cite{SAG}}, where $S^0$ is the 0-sphere. For $E \in \Sp$, $n\in \bZ$ define the $n$-space of $E$ to be $\Omega^{\infty-n}E = \Map_{\Sp}(S^{-n}, E)$. For $X$ a space, the $n$-th cohomology group of $X$ with coefficients in $E$ is defined by $E^n(X) = \pi_0 \Map_{\cS}(X, \Omega^{\infty-n}E)$. Finally for $E \in \CAlg$, define $\pi_nE = E^{-n}(\{x\})$. Armed with this definition, we say $E$ is \textbf{connective} if $\pi_nE= 0 $ if $n<0$. We let $\CAlgcn$ be the full subcategory of $\CAlg$ spanned by the connective $\Einf$-rings. An $\Einf$-ring $E$ is said to be \textbf{discrete} if $\pi_nE = 0$ if $n \neq 0$. We let $\CAlgdiscr$ be the full subcategory of $\CAlg$ spanned by those objects. The construction $A \mapsto \pi_0A$ gives an equivalence of $\infty$-categories between $\CAlgdiscr$ and $\CRing$ (where we take the usual stance as in \cite{HTT} of viewing ordinary categories as $\infty$-categories by taking their nerve). In this manner we can identify a commutative ring with an $\Einf$-ring via this equivalence. Thus we can define $\CAlgk$, and the $\infty$-category of augmented $\Einf$-algebras $(\CAlgk)_{/k} = \CAlgaugk$. $E \in \CAlgaugk$ is said to be \textbf{artinian} if it is connective, $\pi_*E$ is a finite dimensional vector space over $k$, and $\pi_0E$ is a local ring. We denote by $\CAlgartk$ the $\infty$-category of artinian algebras.\\

\section{Spectral Deligne-Mumford stacks}
The fundamental geometric object in this work is that of a  \textbf{Spectral Deligne-Mumford stack} \footnote{Def. 1.4.4.2. \cite{SAG}}, which we will just refer to as spectral DM stack. By definition, a spectrally ringed $\infty$-topos $X = (\cX, \cOX)$ is a spectral DM stack if it has a collection of objects $\Ua \in \cX$ covering $\cX$, such that $\forall \, \alpha$, $\exists \, \Aa \in \CAlg$ along with an equivalence of spectrally ringed $\infty$-topos $(\cXUa, \cOXUa) \simeq \Spet \Aa$, and the structure sheaf $\cOX$ is connective. We now define all the requisite notions.\\

A \textbf{spectrally ringed $\infty$-topos} \footnote{Def. 1.4.1.1 \cite{SAG}} is a pair $(\cX, \cO)$ consisting of an $\infty$-topos $\cX$ and a sheaf $\cO$ of $\Einf$-rings on it.\\

\newpage

For $\cC$ an $\infty$-category, a \textbf{$\cC$-valued sheaf} \footnote{Def. 1.3.1.4} on an $\infty$-topos $\cX$ is a functor $\cX^{\op} \rarr \cC$ that preserves small limits. \\

If $\cOX$ is a $\cC$-valued sheaf on an $\infty$-topos $\cX$, where $\cC$ is an $\infty$-category, $\forall \; U \in \Ob(\cX)$, we denote the composite $(\cXU)^{\op} \rarr \cX^{\op} \xrarr{\cOX} \cC$ by $\cOXU$. Observe that it is also a $\cC$-valued sheaf on $\cXU$ \footnote{Not. 1.4.4.1 \cite{SAG}}.\\

For $R$ an $\Einf$-ring, define the \textbf{\'etale spectrum of $R$} \footnote{Def. 1.4.2.5. \cite{SAG}} by $\Spet R = (\ShetR, \cO)$, where $\ShetR$ is the $\infty$-category $\Sh(\CAlgetR)$ of sheaves of spaces on $\CAlgetR$ for the \'etale topology, and $\cO: \CAlgetR \rarr \CAlg$ is the forgetful functor. Thus defined, $\Spet R$ is a spectrally ringed $\infty$-topos. Observe that $\cO$ is a strictly Henselian sheaf of $\Einf$-rings on the $\infty$-topos $\ShetR$, but it is also a sheaf on $\CAlgetR$ with respect to the \'etale topology. Recall from \cite{SAG} that we have two ways of defining sheaves. For a fixed $\infty$-category $\cC$, $\cX$ an $\infty$-topos, we have $\cC$-valued sheaves on $\cX$ as defined above, and they are objects of the $\infty$-category $\Sh_{\cC}(\cX)$. But if $\cA$ is an essentially small $\infty$-category with a Grothendieck topology, a $\cC$-valued sheaf on $\cA$ \footnote{Def. 1.3.1.1. \cite{SAG}} is a functor $\cF: \cA^{\op} \rarr \cC$ such that $\forall \; U \in \Ob(\cA)$, for any covering sieve $C$ of $U$, we have an equivalence in $\cC$: $\cF(U) \rarr \lim_{V \in C}\cF(V)$. We let $\Sh_{\cC}(\cA)$ be the full subcategory of $\Fun(\cA^{\op}, \cC)$ spanned by the $\cC$-valued sheaves. The connection between these two concepts is provided by Proposition 1.3.1.7 of \cite{SAG}, which states that we have an equivalence of $\infty$-categories: $\Sh_{\cC}(\Sh(\cA)) \simeq \Sh_{\cC}(\cA)$. Letting $\cC = \CAlg$, $\cA = \CAlgetR$, this reads $\Sh_{\CAlg}(\ShetR) \simeq \Sh_{\CAlg}(\CAlgetR)$, from which we see that $\cO$ above can be regarded in two ways.\\

\newpage

Recall from \cite{SAG} that a morphism $\phi:R \rarr S$ in $\CAlg$ is \'etale if $\pi_0R \rarr \pi_0 S$ is \'etale and in addition the morphism $\phi$ induces an isomorphism $\pi_0 S \otimes_{\pi_0 R} \pi_*R \rarr \pi_* S$. We denote by $\CAlgetR$ the full subcategory of $\CAlgR$ consisting of \textbf{\'etale $R$-algebras}. As mentioned above we have a Grothendieck topology on $(\CAlgetR)^{\op}$ defined as follows: \footnote{Prop. B.6.2.1. \cite{SAG}} given $A \in \CAlgR$, a sieve $C$ over $A$ is a covering sieve, iff $C \supseteq \{A \rarr A_i\}_{1 \leq i \leq n}$, with a faithfully flat induced map $A \rarr \prod_i A_i$. We call this the \textbf{\'etale topology} on $(\CAlgetR)^{\op}$ \footnote{Def. B.6.2.2 \cite{SAG}}.\\

\section{Spectral Artin Representability}
One key result for our purposes from \cite{SAG}, is the Spectral Artin Representabiliy Theorem, which states that if $R$ is a Noetherian $\Einf$-ring such that $\pi_0 R$ is a Grothendieck ring, given a natural transformation $p: \mathbf{X} \rarr \Spec R$ in $\Fun(\CAlgcn, \cS)$, given $n \geq 0$, if $\bfX$ is such that $\bfX(A)$ is $n$-truncated for all $A \in \CAlgdiscr$, if $\bfX$ is a sheaf for the \'etale topology, if it is nilcomplete, integrable, infinitesimally cohesive, admits a connective cotangent complex, and $p$ is locally almost of finite presentation, then $\bfX$ is representable by a spectral DM $n$-stack $X$, locally almost of finite presentation over $R$.\\

In this section we will review all the concepts introduced in that theorem.\\

For $R \in \CAlgcn$, we define $\Spec R = \Hom_{\CAlg}(R,-): \CAlgcn \rarr \cS$ \footnote{Not. 6.2.2.3 \cite{SAG}}. If $\infTopsHenCAlg$ denotes the $\infty$-category of spectrally ringed $\infty$-topos with strictly Henselian structure sheaves, then we also have:
\beq
\Spec R = \Hom_{\CAlg}(R,-) = \Hom_{\infTopsHenCAlg}(\Spet(-), \Spet R) \nonumber
\eeq
We can also mention in passing \footnote{Rmrk. 6.2.2.4 \cite{SAG}} that if $\SpDM$ denotes the $\infty$-category of spectral DM stacks, we have a fully faithful embedding $h: \SpDM \rarr \Fun(\CAlgcn, \cS)$, mapping $\Spet R $ to $\Spec R$. That we have such an embedding will allow us later to identify $X \in \SpDM$ with the functor $\bfX = \Hom(\Spet(-), X)$ it represents. For instance, we will regard the manifestation $\bQCoh(\bfX)$ as the manifestation of $X$ through its functor of points. \\

We say that $X \in \SpDM$ is \textbf{locally almost of finite presentation over $R$} \footnote{Def. 4.2.0.1 \cite{SAG}} if for any CD in $\SpDM$ of the form (it being implied $A$ and $B$ are $\Einf$-rings):
\beq
\xymatrix{
	\Spet B \ar[d] \ar[r] & X \ar[d] \\
	\Spet A \ar[r] & \Spet R
}  \nonumber
\eeq
with \'etale horizontal maps, $B$ is \textbf{almost of finite presentation} over $A$ \footnote{Def. 7.2.4.26 \cite{HA}}, that is $B$ is an almost compact object of $\CAlg_A$, meaning $\tau_{\leq n}B$ is a compact object of $\tau_{\leq n} \CAlg_A$ for all $n \geq 0$.\\

Here we are using the notion of \textbf{\'etale map in $\SpDM$} \footnote{Def. 1.4.0.1 \cite{SAG}}. A morphism of spectral DM stacks is in particular a morphism of spectrally ringed $\infty$-topos. A morphism $f:(\cX, \cOX) \rarr (\cY, \cOY)$ between spectrally ringed $\infty$-topos\footnote{Constr. 1.4.1.3 \cite{SAG}} is given by a pair $(f_*, \psi)$, $f_*: \cX \rarr \cY$ a geometric morphism, $\psi: f^*\cOY \rarr \cOX$ the induced map on sheaves in $\Sh_{\CAlg}(\cX)$. We say $f$ is \'etale if $f_* \ind (\cX \simeq \cY_{/U})$, for some $U \in \Ob(\cY)$, and $\psi$ is an equivalence.\\

A functor $\bfX: \CAlgcn \rarr \cS$ is said to be \textbf{nilcomplete} \footnote{Def. 17.3.2.1 \cite{SAG}} if $\forall \, R \in \CAlgcn$, we have a homotopy equivalence $\bfX(R) \xrarr{\simeq} \lim\bfX(\tauleqn R)$.\\

Let $A_1$, $A_2$ and $A_0$ be objects of $\CAlgcn$ such that the maps $\pi_0A_1 \surj \pi_0A_0$ and $\pi_0 A_2 \surj \pi_0 A_0$ have nilpotent kernels in $\pi_0 A_1$ and $\pi_0 A_2$ respectively. Then a functor $\bfX: \CAlgcn \rarr \cS$ is said to be \textbf{infinitesimally cohesive} \footnote{Def. 17.3.1.5 \cite{SAG}} if it maps any pullback CD in $\CAlgcn$ of the form:
\beq
\xymatrix{
	A_{12} \ar[d] \ar[r] &A_2 \ar[d]\\
	A_1 \ar[r] &A_0
} \nonumber
\eeq
to a pullback square in $\cS$:
\beq
\xymatrix{
	\bfX(A_{12}) \ar[d] \ar[r] & \bfX(A_2) \ar[d]\\
	\bfX(A_1) \ar[r] & \bfX(A_0)
} \nonumber
\eeq

Let $R$ be a local Noetherian $\Einf$-ring. Suppose $R$ is complete with respect to its maximum ideal $\gm \subseteq \pi_0 R$. Let $\bfX: \CAlgcn \rarr \cS$ be a functor. If:
\beq
(\Spf R \hrarr \Spec R) \ind  \big( \Map_{\cC}(\Spec R, \bfX)  \xrarr{\simeq} \Map_{\cC}(\Spf R, \bfX) \big) \nonumber
\eeq
where $\cC = \Fun(\CAlgcn, \cS)$, where on the right hand side we have a homotopy equivalence, then we say $\bfX$ is \textbf{integrable} \footnote{Def. 17.3.4.1 \cite{SAG}}. In this definition, $\Spf R$ is the \textbf{formal spectrum} \footnote{Constr. 8.1.1.10 \cite{SAG}}, which is constructed as follows: for $R$ an adic $\Einf$-ring, $I \subseteq \pi_0 R$ a finitely generated ideal of definition, $R \mapsto R^{\wedge}_I$ the $I$-completion functor, denote by $\cOSpfR$ the following composition:
\beq
\CAlgetR \xrarr{\cOSpetR} \CAlgR \xrarr{\wedge} \CAlgR \nonumber
\eeq

Then we define $\Spf R = (\ShadR, \cOSpfR)$. Here $\ShadR$ is the subtopos of $\ShetR$ corresponding to the vanishing locus $X \subseteq | \Spec R|$ of $I$ \footnote{Not. 8.1.1.8 \cite{SAG}}. Going back to the definition of integrability, the inclusion $\Spf R \hrarr \Spec R$ should be understood as $\Hom(\Spet(-), \Spf R) \hrarr \Hom(\Spet(-), \Spet R) = \Spec R$. \\

We now tackle $\bfX$ having a cotangent complex. The full definition is intricate, and rather than being repetitive, we refer the reader to \cite{SAG} for a full coverage. We will limit ourselves to providing the great lines only, since we do not need a working definition, rather we just want to briefly expose the concept of tangent complex to put things in perspective. Let $\bfX: \CAlgcn \rarr \cS$ be a functor, classifiying a L fibration $\overline{\CAlgcn} \rarr \CAlgcn$. Let $\Mod=\Mod(\Sp)$ be the $\infty$-category of pairs $(A,M)$, for $A \in \CAlg$, $M \in \ModA$. Denote $\overline{\CAlgcn} \times_{\CAlg} \Mod$ by $\ModX$, with objects triples $(A,M,\eta)$, with $A$ connective and $\eta \in \bfX(A)$. The full subcategory thereof for which the  objects are such that $M$ is connective is denoted $\ModXcn$\footnote{Not. 17.2.4.1. \cite{SAG}}. Now for two functors $\bfX, \bfY: \CAlgcn \rarr \cS$, $\alpha: \bfX \Lrarr \bfY$ a natural transformation, $\Psi: \ModXcn \rarr \cS$ defined by $\Psi(A,M, \eta) = \bfX(A \oplus M) \rarr \bfX(A) \times_{\bfY(A)} \bfY(A \oplus M)$, let $F = \text{fib}_{\eta}(\Psi)$, for $\eta \in \bfX(A)$. According to Proposition 17.2.3.2 of \cite{SAG}, we have a fully faithful embedding $\iota: \QCoh(\bfX)^{\text{acn}} \hrarr \Fun(\ModXcn, \cS)^{\op}$. Then morally, the \textbf{cotangent complex} \footnote{Def. 17.2.4.2 \cite{SAG}} of $\alpha$ is defined by $\LXY = \iota^{-1}F$. If $\bfY = *$ is the final object of $\FunCAlgcnS$, then we say $\bfX$ admits a cotangent complex $\LX = \text{L}_{\bfX/*}$.\\

One thing we need to clearly define however, that's the notion of quasi-coherent sheaf on stacks and functors. We start with functors, following Construction 6.2.1.7 of \cite{SAG}. Let $q: \cD \rarr \cE$ be a cartesian fibration in $\SetD$. A functor $F: \cC \rarr \cD$ over $\cE$ is said to be $q$-cartesian \footnote{Def. 6.2.1.1. \cite{SAG}} if it maps $\cC_1$ into $q$-cartesian edges of $\cD$. Let $\FuncartECD$ be the full subcategory of $\FunECD$ spanned by such maps. Now let $e:\cC \rarr \cE $ be a R fibration, object of $\SetDRE$, the simplicial category of right fibrations over $\cE$. We have a functor:
\begin{align}
	\text{Fun}: (\SetDRE)^{\op} &\rarr \CatDinf \nonumber \\
	\cC &\mapsto \FuncartECD \nonumber
\end{align}
Here $\CatDinf$ denotes the simplicial category of small $\infty$-categories, with morphisms spaces between $\cC, \cD \in \Catinf$ the largest Kan complex in $\Fun(\cC, \cD)$, hence it is a simplicial category. We have $\text{N}(\CatDinf) = \Catinf$. But we also have $\text{N}(\SetDRE) \simeq \FunEopS$. Thus by taking the simplicial nerve of the functor Fun, we obtain a functor:
\beq
\Phi[q]: \FunEopS^{\op} \rarr \Catinf \nonumber
\eeq
In practice, for $\bfX: \Eop \rarr \cS$ a functor, classifying a R fibration $\cC \rarr \cE$, $\Phi[q](\bfX) = \FuncartECD$ \footnote{Rmrk. 6.2.1.8 \cite{SAG}}. Dually, if $q: \cD \rarr \cE$ is a coCartesian fibration, considering left fibrations over $\cE$ and doing the same construction, we get a functor $\Phi'[q] : \FunEhatS ^{\op} \rarr \widehat{\Catinf}$. We now apply this formalism to:
\beq
q = \pi_1: \cD = \CAlgcn \times_{\CAlg} \Mod \rarr \CAlgcn = \cE \nonumber
\eeq
a coCartesian fibration. Let $\bfX: \CAlgcn \rarr \hat{\cS}$ be a functor, classifying a L fibration $\cC \rarr \CAlgcn$. Then we have:
\begin{align}
	\bQCoh \equiv \Phi'[q]: \FunCAlgcnhatS ^{\op} &\rarr \widehat{\Catinf} \nonumber \\
	\bfX & \mapsto \Fun^{\text{cocart}}_{\CAlg}(\cC, \CAlgcn \times_{\CAlg} \Mod) \nonumber
\end{align}
which defines the \textbf{$\infty$-category of quasi-coherent sheaves on $\bfX$}. In practice though, $\cC$ is denoted $\overline{\CAlgcn}$ and $q: \overline{\CAlgcn} \times_{\CAlg} \Mod \rarr \overline{\CAlgcn} \rarr \CAlgcn$, since we consider functors over $\CAlgcn$. We recognize the fiber product as $\ModX$. Thus we have $\bQCoh(\bfX) = \Fun^{\text{cocart}}_{\CAlg}(\overline{\CAlgcn}, \ModX)$.\\ 

Quasi-coherent sheaves on spectral DM stacks are defined differently. Recall $\Mod$ is the $\infty$-category whose objects are pairs $(A,M)$, $A \in \CAlg$, $M \in \ModA$. We can generalize this as follows: we can consider triples $(\cX, \cO, \cF)$, where $\cX$ is an $\infty$-topos, $\cO$ is a sheaf of $\Einf$-rings on $\cX$, and $\cF$ is a sheaf of $\cO$-module spectra on $\cX$. Let $\infTopMod$ be the $\infty$-category whose objects are such triples, with morphism triples, consisting of a geometric morphism between $\infty$-topoi, and the two corresponding induced maps on sheaves. We can also speak of $(\cO, \cF)$ as being a $\Mod$-valued sheaf on $\cX$. Let $\infTopsHenMod$ be the subcategory of $\infTopMod$ whose objects are such that $\cO$ is strictly Henselian. The global sections functor $\Gamma: \infTopsHenMod \rarr \Mod^{\op}$ has a right adjoint, that we denote by $\Spet_{\Mod}$ \footnote{Cor. 2.2.1.5. \cite{SAG}}.\\

Now for $X = (\cX, \cOX)$ a (non-connective) spectral DM stack, if $\cF$ is a sheaf of $\cOX$-modules on $\cX$, we can regard $(\cX, \cOX, \cF)$ as an object of $\infTopMod$. We say $\cF$ is \textbf{quasi-coherent} \footnote{Def. 2.2.2.1. \cite{SAG}} if there exists a collection $\Ua$ of objects of $\cX$ that cover it, such that $\forall \, \alpha$, $ \exists \,  \Aa \in \CAlg$, $M_{\alpha} \in \Mod_{\Aa}$ with equivalences:
\beq
(\cX/\Ua, \cO|_{\Ua}, \cF|_{\Ua}) \simeq \Spet_{\Mod}(\Aa, M_{\alpha}) \nonumber
\eeq
We denote by $\QCoh(X)$ the full subcategory of $\Mod_{\cOX}$ spanned by such objects $\cF$. \\

\newpage

Finally, for $f: \bfX \Lrarr \bfY$ a natural transformation, for $\bfX, \bfY: \CAlgcn \rarr \cS$, we say $f$ is \textbf{locally almost of finite presentation} \footnote{Def. 17.4.1.1 \cite{SAG}} if for $n \geq 0$, for a filtered diagram $\{\Aa\}$ in $(\CAlgcn)_{\leq n}$ - corresponding to $n$-truncated objects of $\CAlgcn$ - with a colimit $A$, we have a homotopy equivalence $\colim \bfX(\Aa) \rarr \bfX(A) \times_{ \bfY(A)} \colim \bfY(\Aa)$.\\

\section{Representation theory of formal moduli problems}
In this section we introduce formal moduli problems, and one of the main results from \cite{SAG} which we will use: Theorem 13.4.0.1, which states that we have a fully faithful embedding $\bQCohart(\bfX) \hrarr \Repgs$, for $\ggs$ a dg Lie algebra over a field $k$ of characteristic zero, $\bfX = \Psi(\ggs)$ the formal moduli problem associated to $\ggs$.\\

We first define formal moduli problems, then quasi-coherent sheaves defined on them, and finally the morphism $\Psi$ in the above statement.\\

\subsection{Formal moduli problems}
Let $k$ be a field of characteristic zero, $\bfX: \CAlgartk \rarr \cS$ a functor. It is said to be a \textbf{formal moduli problem} \footnote{Ch. 13 \cite{SAG}} if $\bfX(k)$ is contractible, and for objects $R_1, R_2, R_0$ of $\CAlgartk$ such that $\pi_0 R_1 \surj \pi_0 R_0$ and $\pi_0 R_2 \surj \pi_0 R_0$, $\bfX$ maps any pullback:
\beq
\xymatrix{
	R_{12} \ar[d] \ar[r] & R_2 \ar[d] \\
	R_1 \ar[r] & R_0
} \nonumber
\eeq
to a pullback:
\beq
\xymatrix{
	\bfX(R_{12}) \ar[d] \ar[r] & \bfX(R_2) \ar[d] \\
	\bfX(R_1) \ar[r] & \bfX(R_0)
} \nonumber
\eeq
We denote by $\Modulik$ the full subcategory of $\FunCAlgartkS$ spanned by such functors.\\

\subsection{Quasi-coherent sheaves on formal moduli problems}
Quasi-coherent sheaves on formal moduli problems are valued in commutative algebra objects of $\widehat{\Catinf}$, and this can be seen from the following construction.\\

Consider the coCartesian fibration $q:\Mod(\Modk)^{\otimes} \rarr \CAlgk \times \Fins$. There is a map $\xi: \CAlgk \rarr \CAlg(\hCatinf)$ that classifies $q$. Its restriction to artinian algebras admits an essentially unique factorization:
\beq
\CAlgartk \rarr \FunCAlgartkS \rarr \CAlg(\hCatinf) \nonumber
\eeq
where the second map preserves small limits, and is denoted by $\bQCohart$. For $\bfX \in \Modulik$, $\bQCohart(\bfX)$ is called the \textbf{$\infty$-category of quasi-coherent sheaves on $\bfX$} \footnote{Constr. 13.4.6.1 \cite{SAG}}. In this definition, if $\cC$ is a symmetric monoidal $\infty$-category, $\Mod(\Modk)^{\otimes}$ has objects of the form $(A,M_1, \cdots, M_n)$, where $A \in \CAlg(\cC)$, and $M_i \in \ModA$, $1 \leq i \leq n$. $\Fins$ is Segal's category of finite pointed sets with objects of the form $\langle n \rangle$, $n \geq 0$, and morphisms are maps $ \langle m \rangle  \rarr \langle n \rangle $ that preserve the fixed point $*$. \\

Now we can derive a formula for $\bQCohart(\bfX)$ following the method we used to find one for $\bQCoh(\bfX)$, $\bfX: \CAlgcn \rarr \cS$. Consider the projection $q: \overline{\CAlgartk} \times_{\CAlg} \Mod(\Modk) \rarr \overline{\CAlgartk}$. To understand that notation, note that $\Mod(\Modk) = (\Mod(\Modk)^{\otimes})_{<1>}$, it follows $\Mod(\Modk) = \{ (A,M) \in \CAlg(\Modk) \times_{\CAlg} \Mod \}$. Then observe $\CAlg(\Modk) = \CAlg(\Modk(\Sp)) = \CAlgk$ \footnote{Var. 7.1.3.8 \cite{SAG}}, hence $\Mod(\Modk) = \CAlgk \times_{\CAlg} \Mod$, so that we can write $\overline{\CAlgartk} \times_{\CAlg} \Mod(\Modk) = \overline{\CAlgartk}\times_{\CAlg} \Mod$, which we denote by $\ModXart$. Thus $q: \ModXart \rarr \overline{\CAlgartk}$. Let $\bfX: \CAlgartk \rarr \cS$ be a formal moduli problem, classifiying a L fibration $\overline{\CAlgartk} \rarr \CAlgartk$. Then one can show $\bQCohart(\bfX) = \FuncocartCAlgartk(\overline{\CAlgartk}, \ModXart)$.\\

Finally we define the morphism $\Psi$ used in the statement of Theorem 13.4.0.1 of \cite{SAG}. This morphism is actually introduced in Theorem 13.0.0.2 of the same reference. $\Psi: \Liek \rarr \Modulik$ is an equivalence of $\infty$-categories. We have $\Psi(\ggs) = \Map_{\Liek}(\cD(-), \ggs)$, where $\cD: \CAlgaugkop \rarr \Liek$ is the Koszul duality functor, right adjoint to the cohomological Chevalley-Eilenberg complex functor $C^*: \Liek \rarr \CAlgaugkop$. The interested reader will find ample details in Chapter 13 of \cite{SAG}. \\

\section{Dual representation} 
By dual representation, we mean representing a geometric object in two ways. This will involve the $\infty$-category of quasi-coherent sheaves on spectral DM stacks, on objects of $\FunCAlgcnS$, and on formal moduli problems. In a first time, we give the relationship between quasi-coherent sheaves on functors: $\CAlgcn \rarr \cS$, and on formal moduli problems. Then we discuss our first dual representation result at the level of spectral DM stacks, and then at the level of functors $\CAlgcn \rarr \cS$.

\subsection{Relations between quasi-coherent sheaves}
For $\bfX \in \FunCAlgcnS$, let $\isX$ be the pullback of $\bfX$ along the fully faithful embedding $\iota: \CAlgartk \hrarr \CAlgcn$. We say $\bfX: \CAlgcn \rarr \cS$ is \textbf{artinian} if $\isX \in \Modulik$. On the one hand we have $\bQCoh(\bfX) = \FuncocartCAlgcn(\overline{\CAlgcn}, \ModX)$, and on the other we have: 
\beq
\bQCohart(\isX) = \FuncocartCAlgartk(\overline{\CAlgartk}, \ModisXart) \nonumber
\eeq
Consider the following CD:
\beq
\xymatrix{
	\CAlgartk \times_{\cS} \cZ \ar@{=}[d] & \CAlgcn \times_{\cS} \cZ \ar@{=}[d] \nonumber \\
	\overline{\CAlgartk} \ar[d] \ar@{.>}[r]^{\iota \times id} & \overline{\CAlgcn} \ar[d] \ar[r] & \cZ \ar[d] \\
	\CAlgartk \ar@/_2pc/[rr]_{\isX} \ar[r]_{\iota} & \CAlgcn \ar[r]_{\bfX} & \cS
}
\eeq
where $\cZ \rarr \cS$ is the universal left fibration (see \cite{HTT}). The outside square and the rightmost square are Cartesian, hence so is the leftmost square. This means:
\beq
\overline{\CAlgartk} = \CAlgartk \times_{\CAlgcn} \overline{\CAlgcn} \nonumber
\eeq
it follows:
\begin{align}
	\ModisXart &= \overline{\CAlgartk} \times_{\CAlgcn} \Mod \nonumber \\
	&=\big( \CAlgartk \times_{\CAlgcn} \overline{\CAlgcn} \big) \times_{\CAlgcn} \Mod \nonumber \\
	&\simeq \CAlgartk \times_{\CAlgcn} \big( \overline{\CAlgcn} \times_{\CAlgcn} \Mod \big) \nonumber \\
	&= \CAlgartk \times_{\CAlgcn} \ModX \nonumber
\end{align}
Consider the following CD:
\beq
\xymatrix{
	\overline{\CAlgartk} \ar@{=}[d] & \ModisXart \ar@{=}[d] \nonumber \\
	\CAlgartk \times_{\CAlgcn} \overline{\CAlgcn} \ar[d]^{\pi_2} \ar@{.>}[r] & \CAlgartk \times_{\CAlgcn} \ModX \ar[d]_{\pi_2} \ar[r] &\CAlgartk \ar[d]_{\iota} \nonumber\\
	\overline{\CAlgcn} \ar[r] & \ModX \ar[r] & \CAlgcn
}
\eeq

from which we see that $\Fun(\overline{\CAlgartk}, \ModisXart)$ is functorially obtained from $\Fun(\overline{\CAlgcn}, \ModX)$ by base change, which also preserves coCartesian maps. Indeed, recall that if $q: \ModX \rarr \CAlgcn$, a functor $F:\overline{\CAlgcn} \rarr \ModX$ is $q$-coCartesian if it maps every edge of $\overline{\CAlgcn}$ to a $q$-coCartesian edge of $\ModX$. Some of those same edges of $\overline{\CAlgcn}$ are also constituents of edges of $\overline{\CAlgartk} = \CAlgartk \times_{\CAlgcn} \overline{\CAlgcn}$, that map under $\CAlgartk \times_{\CAlgcn} F$ to $\CAlgartk \times_{\CAlgcn} q$-cocartesian edges of $\CAlgartk \times_{\CAlgcn} \ModX = \ModisXart$. Thus we have shown that $\bQCohart(\isX)= \FuncocartCAlgartk(\overline{\CAlgartk}, \ModisXart)$ is functorially obtained from $\bQCoh(\bfX)= \FuncocartCAlgcn(\overline{\CAlgcn}, \ModX)$ by base change. \\

There are a few technical, peripheral results pertaining to quasi-coherent sheaves we go over presently. The first result is the following:
\beq
\big( \bQCoh(\bfX)|_{\CAlgartk} \big)^{\tcn} \supseteq \bQCoh(\bfX)^{\tcn}|_{\CAlgartk} \nonumber 
\eeq
where we have replaced, for notation's sake, the functorial fiber product with $\CAlgartk$ in $\Fun(\overline{\CAlgcn},\ModX)$ with a restriction. Recall from Rmrk 6.2.2.7 of \cite{SAG} that for a functor $\bfX: \CAlgcn \rarr \cS$, for $A \in \CAlgcn$, $\eta \in \bfX(A)$, both of which are encapsulated in a lift $\tilde{A} \in \overline{\CAlgcn}$, where $\overline{\CAlgcn} \rarr \CAlgcn$ is a L fibration classified by $\bfX$, then $\cF \in \bQCoh(\bfX)$ can be seen as a map that to $(A,\eta)$ associates $\cF(\tilde{A}) \in \Mod_A$, and we will use the notation of \cite{SAG}: $\cF(\tilde{A}) = (A,\cF(\eta))$. Now by Def. 6.2.5.3 of \cite{SAG}, if $P$ is a characteristic of pairs $(A,M)$, $A \in \CAlgcn$, $M \in \ModA$, stable under base change, $\cF \in \bQCoh(\bfX)$ is said to have the property $P$ if $\forall \, A \in \CAlgcn$, $\forall \, \eta \in \bfX(A)$, $(A, \cF(\eta))$ has the property $P$. By Prop. 6.2.5.2. of \cite{SAG}, if $P = cn$ is the property of being connective, defined by asking that in a pair $(A,M)$, $M$ is connective seen as a spectrum object, then $cn$ is stable under base change, which allows us to speak of connective quasi-coherent sheaves. To say that $\cF \in \bQCoh(\bfX)$ is \textbf{connective} \index{connective} then means that for all $A \in \CAlgcn$, $\cF(\eta) \in \ModA$ is connective, which is true in particular if $A \in \CAlgartk$, that is $\bQCoh(\bfX)^{\tcn}|_{\CAlgartk} \subseteq \bQCohart(\isX)^{\tcn}$. Now one may very well conceive that $\cF \in \bQCohart(\isX)^{\tcn}$ comes from some object of $\bQCoh(\bfX)$ that was not necessarily connective on non-artinian objects of $\CAlgcn$, thus just have an inclusion. By denoting $\bQCohart(\isX)$ by $\bQCoh(X)|_{\CAlgartk}$, we have the desired result.\\

We will see later that the affine perception of a spectral DM stack $X$ is related to its manifestation by $\AffXop \simeq \CAlg(\bQCoh(\bfX)^{\tcn})$. If we want to connect this with the representation theory aspect of the problem, one may want to relate $\CAlg(\bQCoh(\bfX)^{\tcn})$ with $\CAlg(\bQCohart(\isX)^{\tcn})$. From the work just done on connective sheaves, it is clear that it suffices to establish a relation between $\CAlg(\bQCoh(\bfX))$ and $\CAlg(\bQCohart(\isX))$, if any.\\

Recall from Definition 2.1.2.7 of \cite{HA} that for $\cC$ an $\infty$-operad, $\CAlg(\cC) \subseteq \Fun(\tN(\Fins), \cC)$ is spanned by the $\infty$-operad maps i.e. those maps for which the CD below is commutative: 
\beq
\xymatrix{
	\tN(\Fins) \ar@{=}[dr] \ar[rr] && \cC \ar[dl]\\
	&\tN(\Fins)
} \nonumber
\eeq
and preserves inert morphisms, where by \textbf{inert morphism} \footnote{Def. 2.1.2.3 \cite{HA}} in $\Ctensor$, we mean if $p: \Ctensor \rarr \tN(\Fins)$ is an $\infty$-operad, a morphism $f$ in $\Ctensor$ is inert if $p(f)$ is inert and $f$ is $p$-coCartesian. By Definition 2.1.1.8 of \cite{HA}, $\psi: \langle m \rangle \rarr \langle n \rangle$ a morphism in $\Fins$ is inert if for all $i \in \langle n \rangle -\{*\}$, $\psi^{-1}(i)$ is a singleton. We will split the definition of being inert for a morphism $f$ in $\Ctensor$ into $f$ being \text{$p$-inert}, that is $p(f)$ is inert, and $f$ being $p$-coCartesian.\\

Since we are interested in $\CAlg(\bQCohart(\isX))$ in particular, by the above definitions we should understand inert morphisms in $\bQCohart(\isX)^{\oT}$, so we consider the $\infty$-operad $p: \bQCohart(\isX)^{\oT} \rarr \tN(\Fins)$. Observe that a $p$-inert edge of $p:\bQCoh(\bfX)^{\oT} \rarr N(\Fins)$ (using the same notation $p$ for simplicity) will produce a $p$-inert edge of $\bQCohart(\isX)^{\oT}$ once restricted to artinian algebras, but precisely because of that restriction, there can be edges of $\bQCoh(\bfX)^{\oT}$ that become $p$-inert after restriction, so we have the following inclusion:
\beq
\text{$p$-inert}(\bQCoh(\bfX)|_{\CAlgartk}) \supseteq \text{$p$-inert}(\bQCoh(\bfX))|_{\CAlgartk} \label{pinert}
\eeq
Regarding $p$-coCartesian maps, suppose $f: \cF \rarr \cG$ is $p$-coCartesian in $\bQCoh(\bfX)$. This means $\forall \, \cH \in \bQCoh(\bfX)$, we have a homotopy pullback square:
\beq
\xymatrix{
	\Map_{\bQCoh(\bfX)}(\cG, \cH) \ar[d] \ar[r] &\Map_{\bQCoh(\bfX)}(\cF, \cH) \ar[d] \\
	\Map_{\tN(\Fins)}(p\cG, p\cH) \ar[r] & \Map_{\tN(\Fins)}(p\cF, p\cH)
} \nonumber
\eeq
call the bottom left mapping space $\cC$, the top right one $\cE$, and the bottom right one $\cD$. Then we have $\Map_{\bQCoh(\bfX)}(\cG, \cH) \simeq \cC \times_{\cD} \cE$. Since $\bQCohart(\isX) = \CAlgartk \times_{\CAlgcn} \bQCoh(\bfX)$, if $\cF \in \bQCoh(\bfX)$, we have $\CAlgartk \times_{\CAlgcn} \cF \in \bQCohart(\isX)$. We will denote such restrictions by the same letter $\cF$ for simplicity's sake. Note that we have:
\begin{align}
	\Map_{\bQCoh(\bfX)|_{\CAlgartk}}(\cG, \cH) &= \CAlgartk \times_{\CAlgcn} \Map_{\bQCoh(\bfX)}(\cG, \cH) \label{one} \\
	&\simeq \CAlgartk \times_{\CAlgcn} (\cC \times_{\cD} \cE) \label{two} \\
	&\simeq \cC \times_{\cD} \big( \CAlgartk \times_{\CAlgcn} \cE \big) \label{resfctrCD}\\
	&= \cC \times_{\cD} \Map_{\bQCoh(\bfX)|_{\CAlgartk}}(\cF, \cH) \nonumber
\end{align}
where $\cC$ and $\cD$ in \eqref{resfctrCD} are using restricted functors $p\cF$, $p\cG$ and $p\cH$. Having the above equivalence means precisely that the following diagram is a homotopy pullback square:
\beq
\xymatrix{
	\Map_{\bQCoh(\bfX)|_{\CAlgartk}}(\cG, \cH) \ar[d] \ar[r] &\Map_{\bQCoh(\bfX)|_{\CAlgartk}}(\cF, \cH) \ar[d] \\
	\Map_{\tN(\Fins)}(p\cG, p\cH) \ar[r] & \Map_{\tN(\Fins)}(p\cF, p\cH)
} \nonumber
\eeq
hence $f$ is $p$-coCartesian viewed as a morphism of $\bQCohart(\isX)$. Thus, the restriction to $\CAlgartk$ of a $p$-coCartesian edge of $\bQCoh(\bfX)$ is again a $p$-coCartesian edge of $\bQCoh(\bfX)|_{\CAlgartk}$, but by \eqref{one} and \eqref{two} one can very well have a $p$-coCartesian edge of $\bQCoh(\bfX)|_{\CAlgartk}$ being the restriction to artinian objects of a map in $\bQCoh(\bfX)$ that is not necessarily $p$-coCartesian, thus we have (with obvious notations):
\beq
\text{$p$-coCart}(\bQCoh(\bfX)|_{\CAlgartk}) \supseteq \text{$p$-coCart}(\bQCoh(\bfX)|_{\CAlgartk} \label{pcocart}
\eeq
It follows from \eqref{pinert} and \eqref{pcocart} that we have:
\beq
\CAlg(\bQCohart(\isX)) \supseteq \CAlg(\bQCoh(\bfX))|_{\CAlgartk} \nonumber
\eeq
by formally the same reasoning, we conclude:
\beq
\CAlg(\bQCohart(\isX)^{\tcn}) \supseteq \CAlg(\bQCoh(\bfX)^{\tcn})|_{\CAlgartk} \nonumber
\eeq
Observe that this could also be obtained from the inclusion $\bQCohart(\isX)^{\tcn} \supseteq \bQCoh(\bfX)^{\tcn}|_{\CAlgartk}$, and the fact that $\CAlg(\cC|_{\CAlgartk}) \supseteq \CAlg(\cC)|_{\CAlgartk}$ by following the arguments above.

\subsection{Dual representation for spectral DM stacks}
For $X$ a spectral DM stack, denote by $\AffX$ the full subcategory of $\SpDM$ spanned by affine morphisms $Y \rarr X$, where one says such a map is \textbf{affine} if for any map $\Spet R \rarr X$, $\Spet R\times_X Y$ is affine. We call $\AffX$ the \textbf{affine perception of $X$}. We have the following theorem:

\begin{DualitySpDM}
	Let $X \in \SpDM$, $\bfX$ the functor it represents, which we suppose is artinian. Then we have a dual representation of the manifestation $\bQCoh(\bfX)$ of $\bfX$:
\beq
\xymatrix{
	&\bQCoh(\bfX) \ar[dl]_{\CAlg(-)^{\tcn}} \ar[dr]^{\quad \CAlgartk \times_{\CAlgcn}(-)} \\
	\AffXop && \bQCohart(\isX) \hrarr \Rep_{\Psi^{-1}(\isX)}
} \nonumber
\eeq
\end{DualitySpDM}
\begin{proof}
	By Proposition 2.5.1.2 of \cite{SAG}, we have an equivalence of $\infty$-categories $\AffXop \simeq \CAlg(\QCoh(X)^{\tcn})$. From Proposition 6.2.4.1 of \cite{SAG}, for $X$ a spectral DM stack, we have an equivalence of $\infty$-categories $\QCoh(X) \simeq \bQCoh(\bfX)$, with $\bfX = \Hom(\Spet(-), X)$. Thus $\CAlg(\bQCoh(\bfX)^{\tcn}) \simeq \AffXop$. On the other hand, since $\bfX$ is artinian, $\isX$ is a formal moduli problem, and Theorem 13.4.0.1 of \cite{SAG} provides a fully faithful monoidal embedding $\bQCohart(\isX) \hrarr \Repgs$, where $\ggs = \Psi^{-1}(\isX)$, $\Psi^{-1}$ homotopy inverse to $\Psi$. Further $\bQCohart(\isX)$ is obtained from $\bQCoh(\bfX)$ by base change, which we can formally write as:
	\beq
	\bQCohart(\isX) = \CAlgartk \times_{\CAlgcn}\bQCoh(\bfX) \nonumber
	\eeq
\end{proof}

\subsection{Dual representation for geometric stacks}
In the previous subsection, we introduced the affine perception of a spectral DM stack, the $\infty$-category of all affine maps $Y \rarr X$. If $Y = \Spet R$, we would have a notion of local perception. Observe that if a functor $\bfX: \CAlgcn \rarr \cS$ is represented by $X$, $\bfX = \Hom(\Spet (-), X)$, then $\bfX$ itself provides the \textbf{local perception of $X$}. In this subsection we are interested in recovering $\bfX$ from the quasi-coherent sheaves defined on it. This is possible of $\bfX$, if in addition to being artinian, it is also a \textbf{geometric stack} \footnote{Def. 9.3.0.1 \cite{SAG}}, that is it satisfies descent for the fpqc topology, its diagonal map is affine, and there is a faithfully flat map $\Spec R \rarr \bfX$ for some $R \in \CAlgcn$.

\begin{DualityFctr}
	Let $X \in \SpDM$, representing an artinian, geometric stack $\bfX$. Then we have a dual representation of the manifestation $\bQCoh(\bfX)$ of $\bfX$:
\beq
\xymatrix{
		&\bQCoh(\bfX) \ar[dl]_{\text{Tannaka  Duality}} \ar[dr]^{\quad \CAlgartk \times_{\CAlgcn}(-)} \\
		\bfX = \Hom(\Spet(-),X) && \bQCohart(\isX) \hrarr \Rep_{\Psi^{-1}(\isX)}
		} \nonumber
\eeq
\end{DualityFctr}
\begin{proof}
From Theorem 9.3.0.3 of \cite{SAG}, $\bfX$ a geometric stack can be functorially recovered from $\bQCoh(\bfX)$ by Tannaka duality. The restriction of the latter to artinian algebras is $\bQCohart(\isX)$, full subcategory of $\Repgs$, where $\ggs = \Psi^{-1} (\isX)$, is the $\infty$-category of representations of the dg Lie algebra $\ggs$ associated to $\isX$.
\end{proof}

\subsection{Variants}
We can obtain stronger results, albeit with stricter hypotheses, by asking that $\bfX$ be representable, which necessitates the spectral Artin Representabiliy Theorem.
\begin{Var1}
	Let $\bfX: \CAlgcn \rarr \cS$ be a functor, that satisfies the hypotheses of the spectral Artin representability Theorem for some $n \geq 0$, hence is representable by a spectral DM $n$-stack $X$. Then supposing $\bfX$ is artinian, the manifestation $\bQCoh(\bfX)$ of $\bfX$ has a dual representation:
\beq
\xymatrix{
	&\bQCoh(\bfX) \ar[dl]_{\CAlg(-)^{\tcn}} \ar[dr]^{\quad \CAlgartk \times_{\CAlgcn} (-)} \\
	\AffXop && \bQCohart(\isX)
}\nonumber
\eeq
\end{Var1}

\newpage

\begin{Var2}
	Let $\bfX$ be an artinian, geometric stack that satisfies the hypotheses of the spectral Artin representability Theorem for some $n \geq 0$, let $X$ be the $n$-spectral DM stack representing it. Then the manifestation $\bQCoh(\bfX)$ of $\bfX$ has a dual interpretation:
\beq
\xymatrix{
		&\bQCoh(\bfX) \ar[dl]_{\text{Tannaka  Duality}} \ar[dr]^{\quad \CAlgartk \times_{\CAlgcn}(-)} \\
		\bfX = \Hom(\Spet(-),X) && \bQCohart(\isX) 
		}\nonumber
\eeq
\
\end{Var2}

\end{document}